\RequirePackage{fix-cm}
\documentclass[smallextended]{svjour3}       
\smartqed  
\usepackage{graphicx}
\usepackage{mathptmx} 
\usepackage{
amssymb,
amsmath
}
\usepackage[mathscr]{euscript}
\usepackage{color}
\usepackage{xcolor,colortbl}
\usepackage{multirow}

\spnewtheorem*{theorem*}{Theorem}{\bf}{\it}
\DeclareMathOperator{\coker}{coker}
\DeclareMathOperator{\wt}{wt}

\newcommand{\Z}{\mathbb Z}
\renewcommand{\S}{\mathscr S}
\newcommand{\G}{\Gamma}

\newcommand{\sh}{\cellcolor{black!10}} 
\newcommand{\Q}{\mathbb Q}

\begin{document}

\title{The Sandpile Group of a Thick Cycle Graph
}
\subtitle{
	}


\author{Diane Christine Alar \and
	Jonathan Celaya \and \\ 
	Luis David Garc\'ia Puente \and
	Micah Henson \and
	Ashley K. Wheeler
}

\authorrunning{Alar, Celaya, Garc\'ia Puente, Henson, Wheeler} 

\institute{Diane Christine Alar \at 
	San Francisco State University \\
	Department of Mathematics \\
	Thornton Hall 937 \\
	San Francisco, CA 94132 \\
    	Tel.: (415) 338 2251 \\
    	\email{dalar@mail.sfsu.edu} 
\and Jonathan Celaya \at
	Rice University \\
	Department of Mathematics -- MS 136 \\
	P.O. Box 1892 \\
	Houston, TX 77005-1892 \\
	Tel.: (713) 348-4829 \\
	\email{jsc7@rice.edu}
\and Luis David Garc\'ia Puente \at   
	Sam Houston State University \\
	Department of Mathematics and Statistics \\
	Box 2206 \\
	Huntsville, TX 77341-2206 \\
	Tel.: (936) 294-1581 \\
	\email{lgarcia@shsu.edu}
\and Micah Henson \at
	University of Washington \\
	Department of Applied Mathematics \\
	Box 353925 \\
	Seattle, WA 98195-3925 \\
	Tel.: 404-270-5824\\
	\email{mhenson2@uw.edu}
\and Ashley K. Wheeler \at
	James Madison University \\
	Department of Mathematics and Statistics \\
	60 Bluestone Drive\\
	Harrisonburg, VA 22807 \\
	Tel.: 206-543-5493 \\
	\email{wheeleak@jmu.edu}
}

\date{Received: date / Accepted: date}

\maketitle

\begin{abstract}
The majority of graphs whose sandpile groups are known are either regular or simple.  We give an explicit formula for a family of non-regular multi-graphs called \emph{thick cycles}.  A thick cycle graph is a cycle where multi-edges are permitted.  Its sandpile group is the direct sum of cyclic groups of orders given by quotients of greatest common divisors of minors of its Laplacian matrix.  We show these greatest common divisors can be expressed in terms of monomials in the graph's edge multiplicities. 
\end{abstract}

\section{Introduction}
\label{intro}
The Abelian Sandpile Model was first conceived in 1987 by the physicists Bak, Tang, and Wiesenfeld 
\cite{bak+tang+wiesenfeld}, who developed a cellular automaton model for natural systems with a critical point as an attractor.  In 1990, Dhar \cite{dhar1990} generalized their model to run on a graph with a distinguished vertex called a sink.  The collection of critical stable configurations in the model form a group, the \emph{sandpile group} of a graph.  
%
The sandpile group is also called the \emph{critical group} (in the chip-firing game \cite{biggs1999,bjorner+lovasz,bjorner+lovasz+shor}), or, in other contexts, the \emph{Picard group} or the \emph{Jacobian group} (when regarding the graph as a discrete analogue of a Riemann surface \cite{baker+serguei}), the \emph{group of bicycles} (a way to count spanning trees \cite{berman}), the \emph{tree group} (also related to spanning trees \cite{biggs1997}), and the \emph{group of components} (on arithmetical graphs \cite{lorenzini1989}).  

Explicit formulae for the sandpile groups of several families of graphs are known.  These include complete graphs $K_n$ \cite{biggs1999,lorenzini91}, complete bipartite graphs $K_{m,n}$ \cite{lorenzini91}, complete multipartite graphs $K_{\vec n}$ \cite{jacobson+niedermaier+reiner}, cycles $C_n$ \cite{merris}, generalized de Bruijn graphs \cite{chan+hollmann+pasechnik}, line graphs of graphs \cite{berget+manion+maxwell+potechin+reiner}, M\"obius ladders $M_n$ \cite{chen+hou+woo,deryagina+mednykh}, regular trees \cite{toumpakari}, threshold graphs \cite{christianson+reiner}, square cycles $C^2_n$ \cite{chen+hou+wooSquareCycles}, twisted bracelets \cite{hou+shen}, and wheel graphs $W_n$ \cite{cori+rossin}. Sandpile groups of certain Cartesian products of graphs are also known \cite{chen+hou2008,chen+hou,hou+lei+woo,liang+pan+wang,shi+pan+wang,pan+wang,pan+wang+xu}. However, there are only a few families of non-simple graphs for which the sandpile groups are known; these include thick tree graphs \cite{chen+schedler}, wired tree graphs $\bar T_n$ \cite{levine}, and $(q,t)$-wheel graphs $W_k(q,t)$ \cite{musiker}.

We provide a general formula for the sandpile group of a family of non-regular multi-graphs called \emph{thick cycles} (or \emph{thick $n$-cycles}). This is the first known general formula for the sandpile group of any family of non-regular multi-graphs A thick cycle graph is a cycle with multi-edges.  Section~\ref{sec:preliminaries} of this paper consists of the necessary background information on sandpile groups.  In Section~\ref{sec:mainResults} we prove the following:  

\begin{theopargself}
%
\begin{theorem*}[\ref{thm:group}]
The sandpile group of a thick $n$-cycle $C_{\vec a}$ with multiplicity vector $\vec a=(a_1,\dots,a_n)$ is 
\[
\S(C_{\vec a}) \cong \Z_{g_1} \oplus \Z_{\frac{g_2}{g_1}} \oplus \cdots \oplus \Z_{\frac{g_{n-2}}{g_{n-3}}} \oplus \Z_{\frac{|\S(C_{\vec a})|}{g_{n-2}}},
\]
where $g_t = \gcd{(a_{i_1}\cdots a_{i_t}\mid 1\leq i_1<\cdots <i_t\leq n)}$ for $t=1,\dots,n-2$ ($\gcd{}$ denotes the greatest common divisor).
\end{theorem*}
\end{theopargself}
Theorem~\ref{thm:group} implies the sandpile groups of thick cycles are isomorphic when their multiplicity vectors' entries are permutations of each other.  In Section~\ref{sec:consequences} we list a brief proof of this corollary and a few more consequences.  An explicit formula for the sandpile group of a thick cycle simultaneously gives a formula for the sandpile group of its dual \cite{cori+rossin}.  Thus we have formulae for subdivided banana graphs and we recover and generalize the formula for book graphs given by Emig, et al. \cite{emig+herring+meza+neiuwoudt+gp}.  A formula for the sandpile group of a thick cycle also provides more insight into the study of bilinear pairings on graphs.  We suggest future directions of research to pursue in Section~\ref{sec:future}.    

\section{Preliminaries}
\label{sec:preliminaries}

\begin{definition}
A \emph{thick cycle} of order $n$ (or a \emph{thick $n$-cycle}), denoted $C_{\vec a}$, is a multi-graph consisting of an $n$-cycle with edge multiplicities given by the \emph{multiplicity vector} $\vec a = (a_{1},a_{2},\hdots,a_{n})$. 
\end{definition}

For convention, we label a thick cycle with vertices $v_1,\dots,v_n$ and multiplicity vector $\vec a=(a_1,\dots, a_n)$ such that $a_i$ is the multiplicity of the edge joining $v_i$ and $v_{i+1}$, indices modulo $n$.  Figure~\ref{fig:thickCycleExample} is an example of a thick $5$-cycle with multiplicity vector $\vec a=(3,2,4,2,3)$.    
%
\begin{figure*}
  \includegraphics[width=0.42\textwidth]{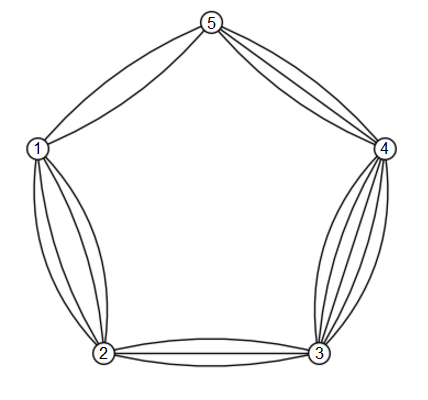}
\caption{Thick $5$-cycle with multiplicity vector $\vec a=(3,2,4,2,3)$.}
\label{fig:thickCycleExample} 
\end{figure*}
Thick cycles are undirected graphs. Recall that an undirected edge can be presented as two opposite directed edges. We note here that the broader sandpile group theory is on directed graphs; however, we shall only consider undirected graphs, in which case the theory is somewhat simplified. 





\subsection{The sandpile group}

The \emph{Laplacian} of a(n undirected) graph $\G$ is the matrix  
\begin{align*}
L=L(\G)=(L_{ij})=\begin{cases}
	-\wt(v_i,v_j) & i\neq j \\
	d_i & i=j,
	\end{cases}
\end{align*}
where $\wt(v_i,v_j)$ denotes the number of edges joining the vertices $v_i$ and $v_j$ and $d_i$ denotes the degree of the vertex $v_i$.  Each of the entries in a given row or sum to zero, and so $L$ has rank at most $n-1$.  Suppose we distinguish a vertex $s=v_i$ in $\G$.  We define the \emph{reduced transposed Laplacian} $\tilde{\Delta}_s$ as the submatrix of $L$ obtained by omitting the $i$th row and column (the Laplacian is symmetric for undirected graphs,but not necessarily for directed graphs).  
%
%
The sandpile group of $\G$ with distinguished vertex $s$ is given by 
\[
\S(\G,s)\cong \Z^{n-1}/\tilde{\Delta}_s\Z^{n-1}=\coker(\tilde{\Delta}_s).
\]  
Surprisingly, it turns out the sandpile group of an undirected graph is independent of the choice of distinguished vertex \cite{cori+rossin}.  Thus we simply write $\S(\G)=\S(\G,s)$.

\subsection{Known results used}

A well-known result in the theory of sandpile groups gives a way to determine the order of $\S(\G)$:    

\begin{theorem*}[Kirchhoff's Matrix-Tree Theorem]
Suppose $\G$ is an undirected graph.  Then for any distinguished vertex $s$ in $\G$, the number of spanning trees of $\G$ is 
\[
\kappa(\G)=|\det(\tilde{\Delta}_s)|,
\]
where $\tilde{\Delta}_s$ is the reduced transposed Laplacian for $\G$.
\end{theorem*}
It follows that the reduced transposed Laplacian has rank exactly $n-1$. \\

Now recall that an $n\times n$ matrix $M$ is {\it $\Z$-equivalent} to a matrix $M'$ if $M'$ can be obtained from $M$ by some sequence of the following row (or column) operations: (1) adding an integer multiple of one row (column) to another, (2) multiplying a row (column) by $-1$, (3) deleting a column (not row) of 0s, or (4) deleting row $i$ and column $j$ if column $j$ is the standard basis vector $e_{ij}$. In addition, given a matrix $M$, we call the nonzero entries on $D$, the diagonal $\Z$-equivalent matrix of $M$, the {\it invariant factors} of $M$. Finally, the {\it cokernel} of $M$, denoted coker($M$), is the quotient group $\Z^n/Im(M)$, where Im($M$) is the image of $M$. Now we may also recall the Invariant Factors Theorem:
\begin{theorem*}[Invariant Factors Theorem]
Suppose $M$ is an $n\times n$ integer matrix of rank $r$.  Then $M$ is $\Z$-equivalent to a diagonal matrix 
\begin{equation}\label{eqn:smithNormalForm}
D=\begin{bmatrix}
f_1 & & & \\
& \ddots & & \\
& & f_r & \\
& & & \mathbf 0
\end{bmatrix}
\end{equation}
and
\[
f_1=m_1,\;f_2=\frac{m_2}{m_1},\;\dots,\;f_r=\frac{m_r}{m_{r-1}},
\]
where $1\leq t\leq r$ denotes the greatest common divisor (gcd) of the $t$-minors of $M$.  
\end{theorem*}
The diagonal matrix $D$ is the \emph{Smith normal form} of $M$.  The non-unit integers among $f_1,\dots, f_r$ are the \emph{invariant factors} of $M$ and of the finitely generated abelian group
\[
\coker(D)\cong \Z_{f_1}\oplus \cdots \oplus \Z_{f_r} \oplus \Z^{n-r}\cong \coker(M).
\]
Thus to compute $\S(\G,s)$, we compute the Smith normal form of $\tilde{\Delta}_s$ (so it is enough to compute the Smith normal form of the Laplacian).  Similar matrices have the same determinant and so $|\det(\tilde{\Delta}_s)|=|\S(\G)|$ is the product of the invariant factors for $\tilde{\Delta}_s$.


\section{Main Results}
\label{sec:mainResults}

\begin{proposition}
\label{prop:order}
Given a thick $n$-cycle $C_{\vec a}$ with multiplicity vector $\vec a=(a_1,\dots,a_n)$, the order of the sandpile group $\S(C_{\vec a})$ is given by the formula
\[
|\S(C_{\vec a})|=\sum_{i=1}^n\frac{a_1\cdots a_n}{a_i}.
\]
\end{proposition}

\begin{proof}
By Kirchhoff's Matrix-Tree Theorem
, the number of spanning trees on a graph is equal to the order of its sandpile group.  To generate a spanning tree for $C_{\vec a}$, we remove the edges between two adjacent vertices and then choose a single edge from each set of edges left. This creates a connected subgraph with $n$ vertices and $n-1$ edges. The number of spanning trees will therefore be the number of ways we can choose $n-1$ edges in this manner.  Let $\G_i$ denote the subgraph of $C_{\vec a}$  obtained by removing the edges between $v_i$ and $v_{i+1}$.  The product of all $a_j$ for $1\leq j \leq n$, $j\neq i$ yields the number of spanning trees on $\G_i$, 
\[
\kappa(\G_i)=\prod_{j=1,j\neq i}^na_j.
\]
The total number of spanning trees on $C_{\vec a}$ is the sum of the number of spanning trees of the subgraphs $\G_i$.
\[\begin{split}
|\S(C_{\vec a})| = \sum_{i=1}^n\kappa(\G_i) &= \sum_{i=1}^n \prod_{j=1,j\neq i}^n a_j =\sum_{i=1}^{n}\frac{a_1\cdots a_n}{a_i}.
\end{split}\]
\qed
\end{proof}

\begin{theorem}[Sandpile Group of a Thick Cycle]
\label{thm:group}
The sandpile group of a thick $n$-cycle $C_{\vec a}$ with multiplicity vector $\vec a=(a_1,\dots,a_n)$ is 
\[
\S(C_{\vec a}) \cong \Z_{g_1} \oplus \Z_{\frac{g_2}{g_1}} \oplus \cdots \oplus \Z_{\frac{g_{n-2}}{g_{n-3}}} \oplus \Z_{\frac{|\S(C_{\vec a})|}{g_{n-2}}},
\]
where $g_t = \gcd{(a_{i_1}\cdots a_{i_t}\mid 1\leq i_1<\cdots <i_t\leq n)}$ for $t=1,\dots,n-2$.
\end{theorem}

\begin{proof}
Unless stated otherwise, when performing arithmetic on any indices we work modulo $n$.  Given the Laplacian matrix $L=L(C_{\vec a})$, let $L'$ denote the matrix resulting from permuting the $j$th column of $L$ to the $(j+1)$th column:  
\[
L'=\begin{bmatrix}
-a_1 & 0 & \cdots & 0 & -a_n & a_n+a_1 \\
a_1+a_2 & -a_2 & \ddots & \ddots & 0 & -a_1 \\
-a_2 & \ddots & \ddots & \ddots & \vdots & 0 \\
\ddots & \ddots & \ddots & -a_{n-2} & 0 & \vdots \\
\ddots & \ddots & -a_{n-2} & a_{n-2}+a_{n-1} & -a_{n-1} & 0 \\
0 & \cdots & 0 & -a_{n-1} & a_{n-1}+a_n & -a_n
\end{bmatrix}
\]
Up to row and column indices, $L'$ and $L$ have the same Smith normal form 
and the same minors.  Thus the invariant factors of $\S(C_{\vec a})$ are
\[
f_1=m_1,\;f_2=\frac{m_2}{m_1},\;\dots,\;f_{n-1}=\frac{m_{n-1}}{m_{n-2}},
\]
where $m_t$ denotes the greatest common divisor of the $t$-minors of $L'$ ($t=1,\dots,n-1$); in fact, $m_{n-1}=|\S(C_{\vec a})|$.  We claim all nonzero $t$-minors of $L'$ are sums of square-free degree $t$ monomials, up to sign, in the multiplicities $a_1,\dots,a_n$.  Granting the claim and using the fact that for integers $a,b$,
\[
\gcd(a,a+b)=\gcd(\pm a,\pm b),
\]
if for every $t$-subset $\{i_1,\dots,i_t\}$ of distinct indices ($t=1,\dots,n-2$) $L'$ has a minor equal to $\pm a_{i_1}\cdots a_{i_t}$, then it follows that $m_t = g_t$. 

It is clear the size $t$ minors are homogeneous of degree $t$ in the $a_i$s because the nonzero entries of $L'$ are all linear in the $a_i$s.  Assume there is a minor with a term that is not square-free.  Then in particular there is a size 2 subminor where an $a_i$ appears either on both diagonal entries or on both anti-diagonal entries.  However, by construction of $L'$ the only such 2-minors are of the form  
\[
\mu=
\begin{vmatrix}
a_{i-1}+a_i & -a_i \\
-a_i & a_i+a_{i+1}
\end{vmatrix},
\]  
in which case the square terms cancel. 

We now show that for a fixed $t$-subset $I=\{i_1,\dots,i_t\}$ of distinct indices, $L'$ has a minor equal to $\pm a_{i_1}\cdots a_{i_t}$.  First, reorder the elements in $I$ so that $i_1<\cdots<i_t\leq n$.  We shall construct a $t\times t$ matrix $M$ similar to a submatrix of $L'$, such that $M$ is block upper triangular, each of its diagonal blocks are either upper or lower triangular, and its main diagonal entries are $-a_{i_1},\dots,-a_{i_t}$.  Then $\det{M}=\pm a_{i_1}\cdots a_{i_t}$ is equal to some $t$-minor of $L'$.

\paragraph{Step 1}
Let $M'$ denote the submatrix of $L'$ given by the row and column indices from $I$.  The main diagonal of $M'$ consists of the entries $-a_{i_1},\dots,-a_{i_t}$ from the main diagonal of $L'$.  We claim that if there exists $i\in I$ such that in $M'$, $-a_i$ is the only nonzero entry in either its row or column, then we may put $M=M'$ and then we are done.  Indeed, if $-a_i$ is the only nonzero entry in its row then $i-1,i-2\notin I$ and we can decompose $M'$ as a block upper triangular matrix 
\begin{equation}
\label{eqn:M'}
M'=\begin{bmatrix}
A & C \\
\mathbf 0 & B
\end{bmatrix},
\end{equation}  
such that $-a_i$ is the upper leftmost entry of $B$, and both blocks $A$ and $B$ are lower triangular.  Then $\det{M}=\det{A}\det{B}$.  On the other hand, if $-a_i$ is the only nonzero entry in its column then $i+1,i+2\notin I$.  Thus we can decompose $M'$ as in (~\ref{eqn:M'}), but with $-a_i$ as the lower rightmost entry of $A$.  Again, both blocks $A$ and $B$ are lower triangular and the result follows.

\paragraph{Step 2}
Given this decomposition, suppose $M'$ has no row or column containing exactly one nonzero entry, i.e., for every $i\in I$, we must have at least one element from each of the sets $\{i-1,i-2\},\{i+1,i+2\}$ also contained in $I$.  It is not clear whether $M'$ has the desired determinant.  We instead use the following algorithm to construct a submatrix of $L'$ whose determinant is $\pm a_{i_1}\cdots a_{i_t}$.  Let $R=(R_1,\dots,R_t)$ and $C=(C_1,\dots,C_t)$ denote ordered, indexed sets.  We initialize by setting $R=C=(i_1,\dots,i_t)$ and putting $k=t$.  At each step, let $M$ denote the submatrix of $L'$ whose rows are indexed by $R$ and columns are indexed by $C$.  If $R_k-R_{k-1}=2$ then replace $R_k\mapsto R_k+1$, $C_k\mapsto C_k-1$, and $k\mapsto k-1$.  The algorithm ends when $k=1$.  

In the algorithm, if the initial consecutive row indices $R_k,R_{k-1}$ differ by 2, it follows that the $k$th column of $M$ consists of $-a_{i_k}$ on the main diagonal with $-a_{i_{k+1}}$ directly below, and this prevents any decomposition into a block upper triangular matrix.  Incrementing $R_k$ by one and decrementing $C_k$ by one alters the indices so that the entry $-a_{i_k}$ on the main diagonal of $M$ comes from the lowermost diagonal of $L'$.  This may cause $C$ to have repeated entries.  However, if that is the case, it means that $R_{k-1},R_{k-2}$ also differ by 2, so again we reselect indices to pick the entry $-a_{k-1}$ from the lowermost diagonal of $L'$.  When the algorithm ends, the resulting matrix $M$ is block upper triangular, with blocks alternating between lower and upper triangular, and its main diagonal consists of the entries $-a_{i_1},\dots,a_{i_t}$.  Therefore $\det{M}=\pm a_{i_1}\cdots a_{i_t}$.
\qed
\end{proof}

\begin{example}
\label{ex:ofTheThm}
Here we provide an example of how to utilize the algorithm in the proof of Theorem \ref{thm:group}. Suppose $n=10$ and we wish to find the minor of the Laplacian $L$ equal to $\pm a_1a_2a_3a_5a_6a_7a_9a_{10}$. 

\paragraph{Step 1} 
The index set is $I=\{1,2,3,5,6,7,9,10\}$ and we have 
\[
M'=\begin{bmatrix}
-a_{1} & 0 & 0 & 0 & 0 & 0 & -a_{10} & a_{10}+a_1 \\ 
a_{1}+a_{2} & -a_{2} & 0 & 0 & 0 & 0 & 0 & -a_1 \\ 
-a_{2} & a_{2}+a_{3} & -a_{3} & 0 & 0 & 0 & 0 & 0 \\ 
0 & 0 & -a_{4} & -a_{5} & 0 & 0 & 0 & 0 \\ 
0 & 0 & 0 & a_{5}+a_{6} & -a_{6} & 0 & 0 & 0 \\ 
0 & 0 & 0 & -a_{6} & a_{6}+a_{7} & -a_{7} & 0 & 0 \\ 
0 & 0 & 0 & 0 & 0 & -a_{8} & -a_{9} & 0 \\ 
0 & 0 & 0 & 0 & 0 & 0 & a_{9}+a_{10} & -a_{10}
\end{bmatrix}.
\]  
However, $M'$ has no row or column consisting of exactly one nonzero entry.  Indeed, Table~\ref{tab:indexConditions} verifies the necessary and sufficient condition on the indices, namely, that for each $k$, at least one element from each of the sets $\{i_k-1,i_k-2\},\{i_k+1,i_k+2\}$ is contained in $I$.  
%
\begin{table}
\caption{Index condition from Example~\ref{ex:ofTheThm}.  For each index $i_k\in I=\{1,2,3,5,6,7,9,10\}$, we check that at least one element from each of the sets $\{i_k-1,i_k-2\},\{i_k+1,i_k+2\}$ is contained in $I$.}
\label{tab:indexConditions} 
\begin{tabular}{llll}
\hline\noalign{\smallskip}
$k$ & $i_k$ & $i_k-1,i_k-2\in I$ & $i_k+1,i_k+2\in I$ \\
\noalign{\smallskip}\hline\noalign{\smallskip}
1 & 1 & \hspace{6mm} 9,10 & \hspace{6mm} 2,3 \\
2 & 2 & \hspace{6.5mm}10,1 &\hspace{7mm} 3 \\
3 & 3 &\hspace{7mm} 1,2 &\hspace{7mm} 5 \\
4 & 5 &\hspace{7.8mm} 3 &\hspace{6mm} 6,7 \\
5 & 6 &\hspace{7.9mm} 5 &\hspace{7mm} 7 \\
6 & 7 & \hspace{8.5mm}6 &\hspace{7mm} 9 \\
7 & 9 &\hspace{8mm} 7 &\hspace{5mm} 10,1 \\
$8=t$ & 10 &\hspace{7.9mm} 9 &\hspace{6mm} 1,2 \\
\noalign{\smallskip}\hline
\end{tabular}
\end{table}
Thus we must proceed to Step 2.

\paragraph{Step 2}
Table~\ref{tab:iterations} shows each iteration of the algorithm that will modify the indices appearing in $R$ and $C$ until the desired matrix $M$ is obtained.  The bold entries indicate each change in $R$ and $C$.  
%
\begin{table}
\caption{Each iteration of the algorithm in Step 2 of Theorem~\ref{thm:group}, applied to Example~\ref{ex:ofTheThm}.}
\label{tab:iterations} 
\begin{tabular}{lllll}
\hline\noalign{\smallskip}
Iteration & $k$ & $R_k-R_{k-1}$ & Resulting $R$ & Resulting $C$ \\
\noalign{\smallskip}\hline\noalign{\smallskip}
1 & 8 & 10-9=1 & \text{no change} & \text{no change} \\
2 & 7 & 9-7=2 & (1,2,3,5,6,$\mathbf 8$,9,10) & (1,2,3,5,6,$\mathbf 6$,9,10) \\
3 & 6 & 8-6=2 & (1,2,3,5,$\mathbf 7$,8,9,10) & (1,2,3,5,$\mathbf 5$,6,9,10) \\
4 & 5 & 7-5=2 & (1,2,3,$\mathbf 6$,7,8,9,10) & (1,2,3,$\mathbf 4$,5,6,9,10) \\
5 & 4 & 6-3=3 & \text{no change} & \text{no change} \\
6 & 3 & 3-2=1 & \text{no change} & \text{no change} \\ 
7 & 2 & 2-1=1 & \text{no change} & \text{no change} \\ 
\noalign{\smallskip}\hline
\end{tabular}
\end{table}
The resulting matrix with row indices $R=(1,2,3,6,7,8,9,10)$ and column indices $C=(1,2,3,4,5,6,9,10)$ is 
\[
M=\left[
\begin{array}{cccccccc}
\sh {\bf -a_1} & \sh 0 & \sh 0 & 0 & 0 & 0 & -a_{10} & a_{10}+a_1 \\
\sh a_1+a_2 & \sh {\bf -a_2} & \sh 0 & 0 & 0 & 0 & 0 & a_1\\ 
\sh -a_2 & \sh a_2+a_3 & \sh {\bf -a_3} & 0 & 0 & 0 & 0 & 0 \\
0&0 & 0 & \sh {\bf-a_5} & \sh a_5+a_6 & \sh -a_6 & 0 & 0 \\ 
0  & 0 & 0 & \sh 0 & \sh {\bf-a_6} & \sh a_6+a_7 & 0 & 0 \\ 
0 & 0 & 0 & \sh 0& \sh 0 & \sh {\bf-a_7} & 0 &0 \\ 
0 &0  & 0 & 0 & 0 & 0 & \sh {\bf-a_9} & \sh 0 \\ 
0  & 0 & 0 & 0 & 0 & 0 & \sh -a_9+a_{10} & \sh {\bf-a_{10}} 
\end{array}
\right].	
\]
The bold entries are the desired factors of the determinant, while the shaded regions are the triangular blocks on the diagonal.  Since below the blocks are zeros the determinant of $M$ is equal to the product of the determinants of the blocks.  
\end{example}

\section{Consequences}
\label{sec:consequences}

In this section, we discuss some consequences of Theorem 1:
\subsection{Permutations of thick cycle multiplicities}

We consider the implication of our result to thick cycles whose edge multiplicities are permutations of each other. That is, we consider when the edge multiplicity vectors of two thick cycles are equivalent up to permutation.

We note that the invariant factors of a thick cycle graph's sandpile group are dependent solely dependent upon the edge multiplicities. Thus, the order in which the edge multiplicities are arranged in our graph has no influence on them. Hence, any permutation of the $n$ edges in an $n$-thick cycle graph will yield the same sandpile group.

\begin{corollary} 
Given a thick cycle $C_{\vec a}$, the sandpile group $\S(C_{\vec a})$ is equal to the sandpile group $\S(C_{\vec b})$, where $\vec b$ is any permutation of the components of $\vec a$.
\end{corollary}

\subsection{Sandpile groups of dual graphs}
\label{subsec:dualGraphs}

A \emph{dual graph} $\G'$ of a planar graph $\G$ is constructed by placing a vertex of the dual in every face of $\G$ and creating edges by connecting vertices of $\G'$ across edges of $\G$. We note that under this definition, a planar graph $\G$ may have many different duals, depending on its embedding. Dual graphs are a generalization of dual tessellations and of dual polyhedra, the latter of which are used in linear and integer programming. R. Cori and D. Rossin showed in 1990 \cite{cori+rossin} that the sandpile groups of a graph and any of its dual are isomorphic.

As an example, recall that the \emph{book graph}, $B(n,k)$, is the graph Cartesian product of the star graph $S_{n+1}$ and the path graph $P_k$.  It has been proved by Emig, et al \cite{emig+herring+meza+neiuwoudt+gp} that $B(n,k)$ is a dual of the subfamily of thick cycle graphs, the thick $(k+1)$-cycles $C_{(1,n-1,\dots,n-1)}$.  Hence, by computing the sandpile group for the general thick cycle graph, we recover and indeed generalize the formula for the sandpile group of book graphs. 

\begin{corollary}
The sandpile group of a book graph $B(n,k)$, with $k$ $n$-cycle pages, is equal to 
\[
\S(B(n,k))= \Z_{n-1}^{k-2}.
\]
\end{corollary} 

Additionally, consider the so-called \emph{subdivided banana graphs} as in \cite{gaudet+jensen+ranganathan+wawrykow+weisman}. A subdivided banana graph is any graph which can be constructed in the following manner: consider two nodes $a$ and $b$ with $k$ edges between them. On each of the edges $1\leq l \leq k$, we introduce $s_l$ new nodes, subdividing the edge from $a$ to $b$ into a path of length $l+2$. This yields the subdivided banana $B_{s_1+1, s_2+1, \hdots, s_k+1}$. The thick cycle graph $C_{\vec a}$ is a planar dual to the subdivided banana graph $B_{\vec s}$, where $\vec a=\vec s$.  Hence, we arrive at the following result:

\begin{corollary}
The sandpile group of a subdivided banana graph $B_{\vec s}$ equals the sandpile group of the thick cycle $C_{\vec s}$.
\end{corollary}

\subsection{Bilinear pairings}
\label{subsec:bilinearPairings}

In this paper, as in Shokreih \cite{shokrieh}, we consider a \emph{bilinear pairing} to be a well-defined, symmetric, non-degenerate map $\langle \cdot,\cdot \rangle: M\times N\rightarrow K$,  where $M, N$ and $K$ are groups. Such a map must satisfy the properties \\
\begin{itemize}
\item $\langle rm,n \rangle = \langle m, rn \rangle = r\langle m, n \rangle$ 
\item $\langle m_1+m_2,n \rangle = \langle m_1, n \rangle + \langle m_2, n \rangle$
\end{itemize}
for any $m, m_1, m_2 \in M$ and $n\in N$.

 On any graph $\G$, the sandpile group $\S(\G)$ comes with a bilinear pairing of the form 
\[
\langle \cdot,\cdot \rangle: \S(\G)\times \S(\G)\to \Q/\Z
\]
with $\langle \cdot,\cdot \rangle$ specifically described by \cite{shokrieh}. This map is called the \emph{monodromy pairing}. Shokrieh uses the monodromy pairing to study the discrete logarithm problem on the Jacobian of finite graphs.  In addition, L. Gaudet, et al. \cite{gaudet+jensen+ranganathan+wawrykow+weisman} have previously used thick cycle graphs to study bilinear pairings as they arise from the sandpile groups of various classes of graphs.

\section{Future Work}
\label{sec:future} 

Thick cycles are one of the first families of multi-graphs, and the first family of non-regular multi-graphs, to have their sandpile group computed. The methods and results in this paper can hopefully be utilized in computing the sandpile groups of more families of non-regular multi-graphs. One such particularly interesting family are the series parallel graphs, of which thick cycles are a subfamily. Series parallel graphs are widely studied in electrical networks and are also researched in computational complexity theory.  Thick cycles have already proved useful in studying this wider class of graphs \cite{noble+royle}. In addition, thick cycle graphs will be of even more use in studying bilinear forms, now that a general formula for their sandpile groups has been computed.

\begin{acknowledgements}
This work was conducted during the 2016 Mathematical Sciences Research Institute Undergraduate Program (MSRI-UP), which is supported by the National Science Foundation (grant No. DMS-1156499) and the National Security Agency (grant No. H98230-116-1-0033). Special thanks to on-site program director Dr. Suzanne Weekes of Worcester Polytechnic Institute.  
\end{acknowledgements}




%
%

\end{document}